\documentclass[12pt]{article}
\usepackage{verbatim}
\usepackage{latexsym}
\usepackage{amsfonts}
\usepackage{amsmath}
\usepackage{amsthm}
\usepackage{amssymb}
\usepackage{url}
\usepackage{authblk}
\usepackage[
    colorlinks=true,
    linkcolor=blue,
    citecolor=blue,
    urlcolor=blue
]{hyperref}
\usepackage{todonotes}
\usepackage{cmbright}
\usepackage{url,tikz}
\usetikzlibrary{arrows,positioning,automata}

\newtheorem{theorem}{Theorem}
\newtheorem{lemma}[theorem]{Lemma}
\newtheorem{corollary}[theorem]{Corollary}

\theoremstyle{definition}
\newtheorem*{definition}{Definition}

\newcommand{\Aut}{\mathrm{Aut}}

\begin{document}
\title{Finite groups with large power-avoiding subsets}
\author[1]{Simon R. Blackburn}
\author[2]{Sarah B. Hart}
\author[2]{Daniel McVeagh}
\affil[1]{Department of Mathematics,
Royal Holloway University of London,
Egham, Surrey TW20 0EX, United Kingdom}
\affil[2]{School of Computing and Mathematical Sciences, Birkbeck College, University of London, Malet Street, London WC1E 7HX, United Kingdom}
\date{}
\maketitle
\begin{abstract}
A subset $X$ of a finite group $G$ is \emph{$k$-power-avoiding} if for all $g\in G$ we have that $\{g,g^k\}\not\subseteq X$. The paper shows that if $G$ contains a $k$-power-avoiding subset $X$ with $|X|\geq |G|-c$, then the group $G^k$ generated by the $k$th powers of elements of $G$ has bounded order (in $k$ and $c$). We provide more detailed structural results when $c\leq 2$, and in particular we classify the groups which arise when $c\leq 2$ and $k$ is prime. 
\end{abstract}

\section{Introduction}

A recurring theme in finite group theory is that the existence of a large subset satisfying a restrictive algebraic condition can impose strong constraints on the structure of the ambient group. We can think of work inspired by the Burnside conjectures~\cite{Burnside} as falling within this theme. Moreover, we may also think of the classical theorem due to B.H. Neumann~\cite{Neumann}, which states that a group $G$ whose centralisers all have finite index bounded by an integer $b$ must have derived subgroup $G'$ of order bounded by a function of $b$. (Guralnick and Mar\'oti~\cite{Guralnick} have provided good bounds on $|G'|$ in this context.) Additional recent examples include the interest in large product-free subsets of finite groups as part of additive combinatorics (see for example, the survey by Kedlaya~\cite{Kedlaya}); a paper by Keevash, Lifshitz and Minzer~\cite{Keevash} resolves the question for the alternating groups). In this paper we study an analogue of product-free subsets, namely subsets of a finite group which avoid containing an element together with a fixed power of that element. 

Let $G$ be a finite group and let $k\geq 2$ be an integer. We say that a subset $X$ of $G$ is \emph{$k$-power-avoiding} if
$$
\{g,g^k\}\not\subseteq X
$$
for every $g\in G$ and write
$$
s_k(G)=\max\{|X|:X\subseteq G\text{ is $k$-power-avoiding}\}.
$$
Thus $s_k(G)$ is the maximum size of a $k$-power-avoiding subset that $G$ can admit. 

What can be said about the structure of a group for which $s_k(G)$ is large? More specifically, suppose that
$$
s_k(G)\geq |G|-c
$$
for some positive integer $c$. Our main result shows that this condition forces $G^k$, the subgroup generated by the $k$th powers of elements of $G$, to have bounded order.

\begin{theorem}\label{thm:main}
Let $G$ be a finite group. Let $k$ and $c$ be positive integers, and suppose
$$
s_k(G)\geq |G|-c.
$$
Then $G$ contains a characteristic subgroup $N$ such that $G/N$ has exponent dividing $k$ and such that $|N|$ is bounded by a function of $c$ and $k$. Indeed, we may take $N=G^k$.
\end{theorem}

The conclusion of Theorem~\ref{thm:main} is essentially best possible. For  suppose that~$G$ has a normal subgroup~$N$ with $|N|\leq c$ and that $G/N$ has exponent dividing $k$. Then $g^k\in N$ for every $g\in G$. Consequently
$$
X=G\setminus N
$$
is $k$-power-avoiding, since $g\in X$ implies that $g^k\in N$ and hence $g^k\notin X$. Thus $s_k(G)\geq |G|-c$. For example, when $k=c=p$ with $p$ prime, one may take $G$ to be an extra-special $p$-group of order $p^{2d+1}$ and set $N=G'$.

We remark that Theorem~\ref{thm:main} is reminiscent of the theorem of B.H. Neumann that was mentioned above: Neumann's theorem says that groups such that all elements commute with much of the group are almost abelian; Theorem~\ref{thm:main} says that 
groups that mainly avoid the $k$th power map are almost of exponent $k$.

Our approach to Theorem~\ref{thm:main} is graph-theoretic. Graphs defined on groups have been studied extensively, and we refer to Cameron~\cite{Cameron} for a general survey. Of relevance here is the power graph of a group, where there is a vertex for each group element, and two vertices are adjacent if one of them is a positive power of the other. This was introduced in the form of a directed graph by Kelarev and Quinn~\cite{Kelarev}, in undirected form by Chakarabarty, Ghosh and Sen~\cite{Chakrabarty} and in papers by Cameron~\cite{Cameron_power} and by Cameron and Ghosh~\cite{CameronGhosh}. We refer to Abawajy, Kelarev and Chowdhury~\cite{AbawajyKelarev} for a survey of power graphs. Cameron and Jafari~\cite{CameronJafari} have studied, in particular, the independence number of the power graph.

More useful for our purposes is a variant of the power graph where we restrict the power concerned to a fixed integer, $k$, and have a directed edge from each group element to its $k$th power:  $g\mapsto g^k$. This digraph is denoted $D_k(G)$ and we note that a $k$-power-avoiding subset of $G$ is then exactly an independent set in $D_k(G)$. Figure~\ref{fig:graph_independent_set} provides an example: the directed graph $D_2(C_{20})$, where $C_{20}$ is the cyclic group of order $20$. We note that $s_k(G)$ is the largest cardinality of an independent set in $D_k(G)$; in the terminology of directed graphs this is known as the independence number of $D_k(G)$. 

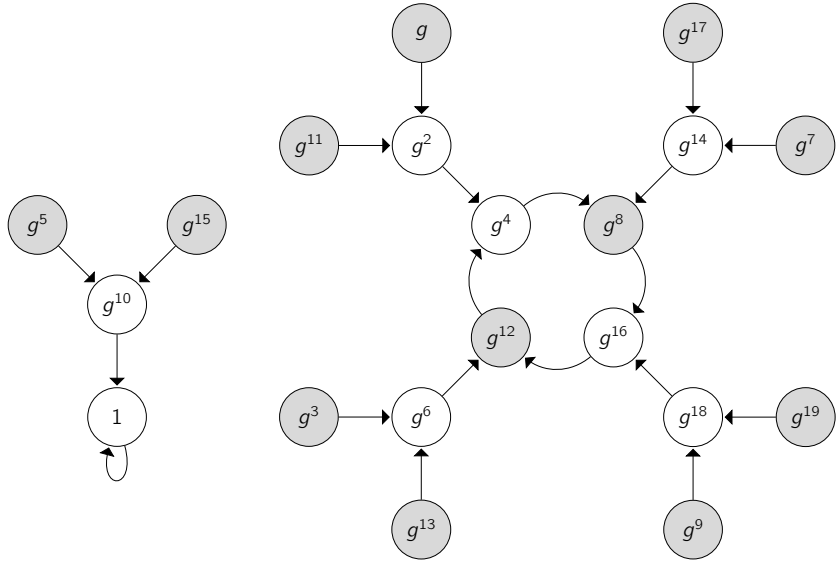
\begin{figure}[ht]
\begin{center}
\scalebox{0.7}{
\begin{tikzpicture}[shorten >=1pt,node distance=1cm,auto,
>=triangle 90]

\node[state] (g4) {$g^4$};
\node[state,fill=gray!30] (g8) [right=of g4] {$g^{8}$};
\node[state] (g16) [below=of g8] {$g^{16}$};
\node[state,fill=gray!30] (g12) [left=of g16] {$g^{12}$};

\node[state] (g14)[above right=of g8] {$g^{14}$};
\node[state,fill=gray!30] (g7)[right=of g14] {$g^{7}$};
\node[state,fill=gray!30] (g17)[above=of g14] {$g^{17}$};

\node[state] (g2)[above left=of g4] {$g^{2}$};
\node[state,fill=gray!30] (g11)[left=of g2] {$g^{11}$};
\node[state,fill=gray!30] (g1)[above=of g2] {$g$};

\node[state] (g6)[below left=of g12] {$g^{6}$};
\node[state,fill=gray!30] (g3)[left=of g6] {$g^{3}$};
\node[state,fill=gray!30] (g13)[below=of g6] {$g^{13}$};

\node[state] (g18)[below right=of g16] {$g^{18}$};
\node[state,fill=gray!30] (g19)[right=of g18] {$g^{19}$};
\node[state,fill=gray!30] (g9)[below=of g18] {$g^{9}$};

\node[state] (g0)[left=2.5cm of g3] {$1$};
\node[state] (g10) [above=of g0] {$g^{10}$};
\node[state,fill=gray!30] (g15) [above right=of g10] {$g^{15}$};
\node[state,fill=gray!30] (g5) [above left=of g10] {$g^5$};

\path[->] (g0) edge [loop below] node {} ();

\draw[->] (g5) -> (g10);
\draw[->] (g10) -> (g0);
\draw[->] (g15) -> (g10);

\path[->] (g4) edge [bend left=40] node {} (g8);
\path[->] (g8) edge [bend left=40] node {} (g16);
\path[->] (g16) edge [bend left=40] node {} (g12);
\path[->] (g12) edge [bend left=40] node {} (g4);

\draw[->] (g2) -> (g4);
\draw[->] (g11) -> (g2);
\draw[->] (g1) -> (g2);

\draw[->] (g14) -> (g8);
\draw[->] (g17) -> (g14);
\draw[->] (g7) -> (g14);

\draw[->] (g18) -> (g16);
\draw[->] (g19) -> (g18);
\draw[->] (g9) -> (g18);

\draw[->] (g6) -> (g12);
\draw[->] (g13) -> (g6);
\draw[->] (g3) -> (g6);

\end{tikzpicture}
}
\caption{The graph $D_2(C_{20})$, where $C_{20}$ is generated by $g$. Shaded vertices form an independent set of maximum possible cardinality, hence $s_k(C_{20})=12$.}
\label{fig:graph_independent_set}
\end{center}
\end{figure}
We note from Figure~\ref{fig:graph_independent_set} that replacing the elements $g^{12}$ and $g^8$ in the independent set shown 
with $g^4$ and $g^{16}$ yields a second $2$-power-avoiding subset of the same cardinality. Hence we do not expect maximum $k$-power-avoiding subsets to be unique in general.

The proof of Theorem~\ref{thm:main} has a graph-theoretic part and a group-theoretic part. We first establish results about directed graphs which possess large independent sets. We then apply these results to $D_k(G)$ to prove Theorem~\ref{thm:main}. The final step uses Zelmanov's solution~\cite{Zelmanov1,Zelmanov2} of the restricted Burnside problem. We refer to Vaughan-Lee~\cite{VaughanLee} for an account of the restricted Burnside problem and its solution.

The rest of the paper is structured as follows. In Section~\ref{sec:notation} we introduce the graph-theoretic terminology that we require. In Section~\ref{sec:directed} we prove the directed-graph results used in our argument. In Section~\ref{sec:main_proof} we apply these results to $D_k(G)$ and complete the proof of Theorem~\ref{thm:main}. In Section~\ref{sec:c=2} we provide (Theorems~\ref{thm:c_equal_1} and~\ref{thm:c=2}) stronger structural results when $c\leq 2$. Indeed, when $k$ is prime we provide a classification of such groups (Corollary~\ref{cor:k_odd_prime} and Theorem~\ref{thm:k=2_classification}). These theorems generalise those in McVeagh~\cite{McVeagh} which contains, among other results, classifications of the finite groups arising when $k=2$ and $c\leq 3$. Finally, in Section~\ref{sec:structure} we show (Theorem~\ref{thm:primary_tree}) that $D_k(G)$ is more symmetrical than you might expect; this result is due to McVeagh~\cite{McVeagh}.

\section{Standard terminology from graph theory}
\label{sec:notation}

This section reminds the reader of (mainly standard) terminology from the theory of finite directed graphs that we use in this paper. We assume the reader is familiar with the standard terminology used in (undirected) graph theory.

A \emph{directed graph} $D$ is a finite set $V$ of vertices, together with a finite set $E$ of directed edges. We assume $E$ is a subset of the set of ordered pairs of vertices. (So there is at most one directed edge in each direction between any two vertices, and we allow a directed loop at a vertex.) We often simplify notation by identifying $D$ with its set of vertices, so $v\in D$ means that $v$ is a vertex of $D$.

We say there is a \emph{directed edge from $u$ to $v$} when $(u,v)\in E$; in this situation we say that $u$ and $v$ are \emph{adjacent}. A \emph{directed walk} in $D$ is a sequence $v_0,v_1,\ldots,v_m$ of vertices such that there is a directed edge from $v_i$ to $v_{i+1}$ for $i\in\{0,1,\ldots,m-1\}$. A directed walk in $D$ is a \emph{directed walk from $v_0$ to $v_m$} if the vertices $v_i$ are distinct. A directed path in $D$ is a \emph{directed cycle} if there is also a directed edge from $v_m$ to $v_0$. When $m=0$, we say that a directed cycle is a \emph{loop}.

Let $v$ be a vertex in a directed graph $D$. The \emph{in-degree} of $v$ is the number of edges of the form $(w,v)$ for some $w\in D$. The \emph{out-degree} of $v$ is the number of edges of the form $(v,w)$ for some $w\in D$. We say that a directed graph $D$ has \emph{out-degree 1} if all vertices $v\in D$ have out-degree $1$. An example is the directed graph $D$ obtained from a function $f\colon D\rightarrow D$ by adding a directed edge from $v$ to $f(v)$ for all $v\in D$. When $D$ has out-degree~$1$, we write $[v]$ for the unique infinite directed walk starting at a vertex~$v$.

A vertex $v$ is an \emph{isolated vertex} if $v$ has in- and out-degree $0$. So an isolated vertex is involved in no directed edges.

A vertex $v$ in a directed graph $D$ is a \emph{leaf} if it has in-degree~$0$ and out-degree at most $1$. (Note that a vertex with in-degree $1$ and out-degree $0$ is not a leaf according to our terminology, but an isolated vertex is a leaf.)

Let $D$ be a directed graph and let $r\in D$. We say that $D$ is a \emph{root-directed tree with root $r$}  if $D$ has no directed cycles, the vertices in $D\setminus\{r\}$ have out-degree $1$ and the vertex $r$ has out-degree $0$. The vertex $r$ is the \emph{root}. Since the root-directed tree $D$ is finite (because by definition every directed graph is finite), there is a unique directed path from each vertex in $D\setminus\{r\}$ to the root $r$.  The graph $D$ is a \emph{root-directed forest} if $D$ is the disjoint union of a finite number of root-directed trees. Note that any nonempty subgraph of a root-directed forest is itself a root-directed forest.

An \emph{independent set} $X$ in a directed graph $D$ is a subset of vertices that does not contain a directed edge: if there is an edge from $u$ to $v$ then either $u\notin X$ or $v\notin X$. In particular, an independent set $X$ does not contain any vertices where there are loops.

\section{Directed graphs}
\label{sec:directed}

This section contains two combinatorial lemmas concerning directed graphs. The first is a well-known exercise; the second is the key combinatorial result in this paper.

\begin{lemma}
\label{lem:path}
Let $n$ be a positive integer. Define $P_n$ to be the directed graph on $n$ vertices $v_0,v_1,\ldots ,v_{n-1}$ with directed edges from $v_i$ to $v_{i+1}$ for $0\leq i<n-1$.
The maximum cardinality of an independent set $X$ in $P_n$ is $\lceil n/2\rceil$.
\end{lemma}
\begin{proof}
We prove the lemma by induction on $n$. When $n=1$ or $n=2$, the result is trivial. Suppose, as an inductive hypothesis, that the lemma holds for $P_n$ with $n\leq s$, for some integer $s\geq 2$. Let $X$ be an independent set in $P_{s+1}$. We need to show that $|X|\leq \lceil (s+1)/2\rceil$.
If $v_s\notin X$, then $X$ is an independent set in $P_s$. So by our inductive hypothesis, $|X| \leq \lceil s/2\rceil\leq \lceil (s+1)/2\rceil$, as required.
If $v_s\in X$, then $v_{s-1}\notin X$ (as there is an edge from $v_{s-1}$ to $v_s$, and $X$ is an independent set). Hence $X=\{v_s\}\cup X_1$, where $X_1$ is an independent set in $P_{s-1}$. By our inductive hypothesis,
\[
|X|=1+|X_1|\leq 1+\lceil (s-1)/2\rceil=\lceil (s+1)/2\rceil.
\]
So the lemma holds for $P_{s+1}$. The lemma now follows by induction on~$s$.
\end{proof}
We observe that the bound in Lemma \ref{lem:path} is attained by (for example) the set $\{v_i: i{\text{ is even}}\}$. 
\begin{lemma}
\label{lem:forest_nonleaves}
Let $c$ be a non-negative integer. Let $D$ be a root-directed forest. Suppose $D$ contains an independent set of cardinality $|D|-c$. Then
there are at most $2c$ non-leaves in $D$.
\end{lemma}
\begin{proof}
We will prove the result by induction on $c$. For the base case $c=0$, note that when $D$ is, itself, an independent set there are no edges in $D$. So $D$ is a set of isolated vertices and thus every vertex is a leaf, as required.

Suppose, as an inductive hypothesis, that $c>0$ and the lemma holds for all smaller values of $c$.

Let $L$ be the set of leaves of $D$. Let $M$ be the set of vertices of $D$ that are adjacent to one or more leaves, so a vertex $w$ lies in $M$ if and only if there is a directed edge from $v$ to $w$ for some leaf $v\in L$. (See Figure~\ref{fig:root_directed_tree} for an example.) Note that every vertex $w\in M$ has in-degree at least $1$, and so $L\cap M=\emptyset$. Define $m=|M|$.

\begin{figure}
\begin{center}
\scalebox{0.5}{
\begin{tikzpicture}[shorten >=1pt,node distance=1cm,auto,
>=triangle 90]

\node[state] (p1) {};
\node[state] (p2)[above=of p1] {};
\node[state,fill=gray!30] (p3)[above=of p2] {};
\node[state] (p4)[above=of p3] {};
\node[state,fill=gray!30] (p5)[above=of p4] {};

\node[state] (e1)[ right=2.5cm of p2] {};

\node[state,fill=gray!30] (m3)[left=1.5cm of p3] {};
\node[state,fill=gray!30] (m4)[right=5.5cm of p3] {};

\node[state,fill=gray!90] (L10)[left=0.8cm of p4] {};
\node[state,fill=gray!30] (m2)[left=0.4cm of L10] {};
\node[state,fill=gray!90] (L12)[right=1.6cm of p4] {};
\node[state,fill=gray!90] (L11)[left=0.2cm of L12] {};
\node[state,fill=gray!90] (L13)[right=0.2cm of L12] {};

\node[state,fill=gray!90] (L14)[right=0.7cm of L13] {};
\node[state,fill=gray!90] (L15)[right=0.4cm of L14] {};

\node[state,fill=gray!30] (m1)[right=2.1cm of p5] {};

\node[state,fill=gray!90] (L3)[above=of p5] {};
\node[state,fill=gray!90] (L2)[left=0.2cm of L3] {};
\node[state,fill=gray!90] (L4)[right=0.2cm of L3] {};

\node[state,fill=gray!90] (L6)[above=of m1] {};

\node[state,fill=gray!90] (L9)[left=1.6cm of p5] {};
\node[state,fill=gray!90] (L8)[left=0.4cm of L9] {};

\draw[->] (L3) -> (p5);
\draw[->] (p5) -> (p4) ;
\draw[->] (p4) -> (p3);
\draw[->] (p3) -> (p2);
\draw[->] (p2) -> (p1);
\draw[->] (L2) -> (p5);
\draw[->] (L4) -> (p5);

\draw[->] (e1) -> (p1);
\draw[->] (m4) -> (e1);
\draw[->] (L14) -> (m4);
\draw[->] (L15) -> (m4);

\draw[->] (m3) -> (p2);
\draw[->] (m2) -> (m3);
\draw[->] (L10) -> (m3);
\draw[->] (L8) -> (m2);
\draw[->] (L9) -> (m2);

\draw[->] (L11) -> (p3);
\draw[->] (L12) -> (p3);
\draw[->] (L13) -> (p3);

\draw[->] (m1) -> (p4);
\draw[->] (L6) -> (m1);

\end{tikzpicture}
}

\caption{A root-directed tree, with leaves (the set $L$) shaded darkly, and vertices adjacent to leaves (the set $M$) shaded lightly. When the shaded vertices (in $L$ and $M$) are removed, we obtain a root-directed forest $D_0$ made up of two root-directed trees, with $1$ and $3$ vertices respectively.}
\label{fig:root_directed_tree}
\end{center}
\end{figure}

If $m=0$, all leaves have out-degree $0$ and so the directed forest $D$ is a set of isolated vertices: there is nothing to prove. So we may assume that $m>0$. 

Let $X$ be an independent set in $D$ of maximal size. Further, let $X$ be chosen to contain the largest number of leaves in $D$, subject to $X$ being an independent set of maximal size. We claim that $L\subseteq X$. To see this, let $v\in L$ be a leaf, so $v$ has in-degree $0$, and suppose for a contradiction that $v\not\in X$. If $v$ is a root, then $v$ has out-degree $0$ as well as in-degree~$0$: no directed edges involve $v$.  So $X\cup\{v\}$ is an independent set, which contradicts the maximality of $X$. Hence $v$ is not a root, and there is a directed edge from $v$ to $w$ for some $w\in M$. Note that $w\notin L$, as $L\cap M=\emptyset$. If $w\not\in X$, then $X\cup\{v\}$ is an independent set and again we have a contradiction as $X$ is a maximal independent set. So $w\in X$. But then $(X\setminus\{w\})\cup\{v\}$ is an independent set of maximal size containing more leaves than $X$. This contradicts our choice of $X$. So $v\in X$ for any $v\in L$ and thus $L\subseteq X$, as required.

The independent set $X$ above has cardinality at least $|D|-c$, and has the property that $L\subseteq X$. If $w\in M$, then there exists a directed edge from a leaf $v\in L$ to $w$. Since $X$ is an independent set containing $v\in L$, we deduce that $w\notin X$. Hence $M$ is contained in the complement of $X$ in $D$, and so $m=|M|\leq c$.

Let $D_0$ be the root-directed forest obtained by deleting the vertices in $L\cup M$. Every leaf $v$ in $D_0$ is at the end of a directed edge in $D$ from some vertex $z\in M$. (To see this, note that there must be a directed edge $(z,v)$ in $D$ for some vertex $z\in D$, since $v\notin L$. And $z\notin L$, since $v\notin M$. If $z\notin M$, then $z\in D_0$ and so $(z,v)$ is an edge in $D_0$. This contradicts the fact that $v$ is a leaf in $D_0$. So $z\in M$.) Since all vertices in $M$ have out-degree at most~$1$, the number of leaves in $D_0$ is at most $m$. 

The set $X_0=X\cap D_0$ is an independent set in $D_0$. Since $X$ contains $L$ but is disjoint from $M$, we see that
\begin{align*}
|X_0|&=|X|-|L|\geq |D|-c-|L|\\
&=|D_0|+m+|L|-c-|L|\\
&=|D_0|-(c-m).
\end{align*}
Since $m$ is positive, $c-m<c$. By our inductive hypothesis, $D_0$ contains at most $2(c-m)$ non-leaves. Since $D_0$ contains at most $m$ leaves, $|D_0|\leq 2(c-m)+m=2c-m$. But $D_0\cup M$ is the set of non-leaves in $D$, and so the number of non-leaves in $D$ is
\[
|D_0|+|M|\leq 2c-m+m=2c,
\]
establishing our inductive step. So the lemma now follows.
\end{proof}

We remark that the directed path on $2c+1$ vertices shows that Lemma~\ref{lem:forest_nonleaves} is tight.

\section{Proof of Theorem~\ref{thm:main}}
\label{sec:main_proof}

\begin{lemma}
\label{lem:exponent}
Let $k$ and $c$ be positive integers. Let $G$ be a finite group containing a $k$-power-avoiding set $X$ of cardinality at least $|G|-c$. Then the exponent of $G$ divides $k^{2c}\prod_{i=1}^{2c}(k^i-1)$. In particular, the exponent of $G$ is bounded by a function of $c$ and $k$.
\end{lemma}
\begin{proof}
Let $g\in G$. Define a sequence $g_0,g_1,\ldots ,$ of elements of $G$ by $g_0=g$ and $g_{i+1}=g_i^k$ for all $i\geq 0$. Let $P=\{g_i:i\geq 0\}\subseteq G$. Define $n=|P|$. 

Add a directed edge from $g_i$ to $g_{i+1}$ for $0\leq i<n-1$, so $P$ becomes a directed path. Since $X$ is $k$-power-avoiding and all edges in $P$ are from $x\in P$ to $x^k\in P$, we see that $X\cap P$ is an independent set in $P$. Since $|X|\geq |G|-c$, at most $c$ elements of $G$ lie outside $X$ and so at most $c$ elements of $P$ lie outside $X$. Hence $|X\cap P|\geq |P|-c=n-c$. By Lemma~\ref{lem:path}, the independent set $X\cap P$ of $P$ has cardinality at most $\lceil n/2\rceil$. Hence $n-c\leq \lceil n/2\rceil$ and therefore $n\leq 2c+1$.

By the pigeonhole principle, the sequence $g_0,g_1,\ldots ,g_n$ of elements of $P$ must contain two equal elements: $g_\ell=g_m$ for integers $\ell$ and $m$ where  $0\leq \ell<m\leq n$. Hence $g^{k^\ell}=g^{k^m}$, and the order of $g$ divides $k^m-k^\ell$. Since $k^m-k^\ell$ divides $k^{n-1}(k^{m-\ell}-1)$, the order of $g$ divides $k^{n-1}\prod_{i=1}^{n-1}(k^i-1)$.  Since $n\leq 2c+1$, we see that the order of $g$ divides  $k^{2c}\prod_{i=1}^{2c}(k^i-1)$  for all $g\in G$, and so the exponent of $G$ divides $k^{2c}\prod_{i=1}^{2c}(k^i-1)$ as required.
\end{proof}

\begin{lemma}
\label{lem:nonleaf}
Let $k$ and $c$ be positive integers. Let $G$ be a finite group containing a $k$-power-avoiding set $X$ of cardinality at least $|G|-c$. Then the directed $k$th power graph $D_k(G)$ of $G$ has at most $3c$ non-leaves.
\end{lemma}
\begin{proof}
The set $X$ is an independent set of $D_k(G)$, and at most $c$ vertices in $D_k(G)$ lie outside $X$.

Since every vertex of $D_k(G)$ has out-degree $1$, we see that all directed cycles in $D_k(G)$ are disjoint, and the only cycles in $D_k(G)$ are directed cycles.

Suppose that $D_k(G)$ contains $z$ directed cycles. No cycle can be entirely contained in an independent set, and so there are at
least $z$ elements of $D_k(G)$ which lie outside $X$. Hence $z\leq c$.

Order the $z$ directed cycles in $D_k(G)$ in some arbitrary way. Choose vertices $r_i$ and $s_i$ in the $i$th cycle (where $1\leq i\leq z$) such that there is a directed edge from $r_i$ to $s_i$. (We have $r_i=s_i$ when the cycle is a loop.)  Let $F$ be the directed graph $D_k(G)$ with the directed edges $r_i$ to $s_i$ removed for $1\leq i\leq z$. We see that the elements $r_i$ have out-degree $0$ in $F$, and all other vertices have out-degree $1$.  Since $D_k(G)$ is finite and all vertices have out-degree~$1$ in $D_k(G)$, the unique directed walk in $D_k(G)$ starting at any vertex ends up traversing a cycle; terminating this walk at the vertex $r_i$ on the cycle, we deduce that there is a unique directed path in $D_k(G)$ from any vertex to one of the vertices~$r_i$. Note that the edge from $r_i$ to $s_i$ is not used in this path, so the path is in $F$. For each vertex $r_i$, the set of vertices with directed paths to $r_i$ forms a root-directed tree with root $r_i$. Hence $F$ is a root-directed forest, with roots $r_1,r_2,\ldots,r_z$. 

The set $X$ is an independent set in $F$, as we have removed edges to construct $F$ from $D_k(G)$. By Lemma~\ref{lem:forest_nonleaves}, $F$ contains at most $2c$ non-leaves. A non-leaf in $D_k(G)$ is either a non-leaf in $F$, or is one of the vertices~$s_i$. Since there are $z$ vertices $s_i$, the directed graph $D_k(G)$ has at most $2c+z$ non-leaves. Since $z\leq c$, we see that  $D_k(G)$ has at most $3c$ non-leaves, as required. 
\end{proof}

\begin{proof}[Proof of Theorem~\ref{thm:main}]
Let $k$ and $c$ be positive integers. Let $G$ be a finite group containing a $k$-power-avoiding set $X$ of cardinality at least $|G|-c$.

Define the subset $Y=\{g^k:g\in G\}$. The elements of $Y$ are precisely the non-leaves in $D_k(G)$, and so $|Y|\leq 3c$ by Lemma~\ref{lem:nonleaf}. Let $N$ be the subgroup of $G$ generated by $Y$. Since $Y$ is a characteristic subset of $G$, we see that $N$ is a characteristic subgroup of $G$. Moreover, $N$ contains all $k$th powers in $G$, and so $G/N$ has exponent dividing $k$.

Since $N$ is a subgroup of $G$, $N$ is finite and the exponent of $N$ divides the exponent of $G$. The exponent of $N$ is bounded as a function of~$c$ and~$k$ by Lemma~\ref{lem:exponent}. The subgroup $N$ is generated by at most $3c$ elements. By Zelmanov's solution~\cite{Zelmanov1,Zelmanov2} to the restricted Burnside problem (see~\cite{VaughanLee} or~\cite{Zelmanov3}), the order of $N$ is bounded by a function of $c$ and $k$. So the theorem follows.
\end{proof}

Zelmanov's solution to the restricted Burnside problem is a deep result. His proof for general exponents uses Hall and Higman's reduction~\cite{HallHigman} to the case of prime power exponents; the classification of finite simple groups is currently needed for this reduction to be applied. A natural question is: can Theorem~\ref{thm:main} be proved more directly?

\section{When $s_k(G)$ is large}
\label{sec:c=2}

For a given $c$ we can ask about the groups $G$ for which $s_k(G)=|G|-c$. The $c=1$ case is easy for every $k$: the only way it can happen is for $X=G\setminus\{1\}$, and that is possible if and only if $\exp(G)$ divides $k$. We state this as a theorem.
\begin{theorem}
    \label{thm:c_equal_1}
    Let $k$ be a positive integer. A finite group $G$ has $s_k(G)=|G|-1$ if and only if the exponent $\exp(G)$ of $G$ divides $k$.
\end{theorem}

The rest of this section looks at the case when $s_k(G)=|G|-2$. We first (Theorem~\ref{thm:c=2}) provide a classification for groups which occur in this case. In the following subsection, we give a more detailed classification (Corollary~\ref{cor:k_odd_prime} and Theorem~\ref{thm:k=2_classification}) when $k$ is prime.

\subsection{The case $s_k(G)=|G|-2$}

\begin{theorem}\label{thm:c=2}
Let $G$ be a finite group and $k$ a positive integer. Then $s_k(G)=|G|-2$ if and only if one of the following holds.
\begin{enumerate}
	\item[(a)] $k$ is odd and $G\cong H\times C_2$, where $H$ is any group of exponent dividing $k$. (In this case $G$ contains a unique involution $g$, and $G\setminus\{1,g\}$ is $k$-power-avoiding.)
	\item[(b)] $k\equiv 2\mod{3}$ and $G$ is cyclic of order 3. (Any non-identity element of $G$ is a suitable $k$-power-avoiding set.)
	\item[(c)] $k\equiv 2\mod{6}$ and $G$ is dihedral of order 6. (If $g$ is any element of order 3 in $G$, then $G\setminus\{1,g\}$ is $k$-power-avoiding.)
	\item[(d)] $k$ is even, $G$ is a central extension of a group $H$, of exponent dividing $k$, by $C_2$, and $\exp(G)$ does not divide $k$. (In this case there is a central subgroup $\{1,g\}$, and $G\setminus\{1,g\}$ is $k$-power-avoiding.)
\end{enumerate}
\end{theorem}

We will need the following standard lemma for the proof of Theorem~\ref{thm:c=2}.

\begin{lemma}\label{lem:twiceodd}
Suppose $G$ is a group of order $2m$ where $m$ is odd. If $G$ has a central involution, then $G$ is isomorphic to the direct product of a group of order $m$ with the cyclic group of order 2.
\end{lemma}

\begin{proof}
The action of $G$ on itself by left multiplication induces a homomorphism $f$ from $G$ to $S_G$. Combining this with the standard homomorphism from $S_G$ to $\{1,-1\}$ given by $\sigma\mapsto\text{sgn}(\sigma)$ for $\sigma\in S_G$ gives a homomorphism $\theta:G\rightarrow\{1,-1\}$. Let $g$ be a central involution in $G$. Its image $f(g)$ in $S_G$ consists of $m$ 2-cycles $(x,gx)$. Since $m$ is odd, we see that $\text{sgn}(f(g))=-1$. Hence $\operatorname{im}(\theta)=\{1,-1\}$. Therefore $\ker(\theta)$ is a subgroup $H$ of index 2 in $G$. Since $H$ has order $m$, and $m$ is odd, we have that $H$ and $\langle g\rangle$ are normal subgroups with trivial intersection whose product is $G$. Therefore $G\cong H\times\langle g\rangle\cong H\times C_2$, as required.
\end{proof}

\begin{proof}[Proof of Theorem~\ref{thm:c=2}]
Suppose $s_k(G)=|G|-2$. Then $\exp(G)$ does not divide~$k$. Therefore, either $\exp(G)$ has a prime factor $p$ which does not divide $k$, or there is a prime $p$ and a positive integer $\ell>1$ such that $p^{\ell-1}$ divides $k$, $p^\ell$ does not divide $k$, and $p^\ell$ divides $\exp(G)$.

Assume first that $\exp(G)$ has a prime factor $p$ which does not divide $k$. Let $x$ be an element of order $p^r$ in $G$ ($r\geq 1$). At least one such element exists.
Since $k$ and $p$ are coprime, $x^{k^i}$ also has order $p^r$, for all $i$. Thus $x$ lies in a directed cycle of $D_k(G)$, and this cycle is not the identity loop. Therefore $D_k(G)$ contains at least two directed cycles. Since no cycle can lie entirely in an independent set, there must be exactly two cycles. One is the loop containing the identity element. The remaining cycle $C$ must contain not just $x$ but all non-identity elements of $p$-power order. Moreover, $C$ must also contain all non-identity elements of $p'$-power order for any other prime factors $p'$ of $\exp(G)$ that do not divide $k$. But elements of $C$, because $C$ contains $x$, must all lie in $N:=\langle x\rangle$ and have the same order as $x$. Thus $p$ is the unique prime factor of $\exp(G)$ (and hence of $|G|$) that does not divide $k$. Since $G$ (and therefore $C$) must contain an element of order $p$, we see that $x$ must have order $p$. Therefore $N\cong C_p$ and in fact $N$ is the unique Sylow $p$-subgroup of $G$. In particular $\exp(G)=pe$, where $e$ is coprime to $p$, and all prime factors of $e$ divide $k$. The non-identity elements of $N$ constitute a single cycle of $D_k(G)$, and any $k$-power-avoiding set $X$ in $G$ of cardinality $|G|-2$ must be $G\setminus\{1,g\}$ for some element $g$ of this cycle (and note that $\langle g\rangle=N$).

Let $t$ be a prime-power factor of $e$ and let $y$ be an element of order $t$. One of $y$ and $y^k$ must lie outside $X$. As all elements in $C$ have order $p$, we see that $y\not\in C$ and $y^k\not\in C$. Since $y\not=1$, we must have $y^k=1$ and so $t$ divides $k$. We deduce that $e$ divides $k$. 

At most half the elements of any directed cycle are contained in any independent set. Therefore $C$ has at most two vertices.

Suppose that $C$ has exactly one vertex. Then $p=2$, $k$ is odd, there is a unique involution (which is therefore central), and $|G|=2m$ where $m$ is odd. By Lemma~\ref{lem:twiceodd}, $G\cong H\times C_2$ where $|H|=m$. We have $\exp(G)=2\exp(H)$, and so $\exp(H)=e$. Since $e$ divides $k$, this is case (a) of the theorem.

Suppose now that $C$ has two vertices. Then $p=3$, the vertices are the non-identity elements of $N$, namely $g$ and $g^2$, and so $g^k=g^2$. That is, $k\equiv 2\mod{3}$. Since $N$ is normal, the conjugacy class of $g$ contains at most two elements. Thus $C_G(g)$ has index at most 2 in $G$. Suppose $C_G(g)\neq\langle g\rangle$. Then $|C_G(g)|$ has a prime factor $q\neq 3$. Let $h$ be an element of order $q$ in $C_G(g)$. Since all primes other than $3$ that divide $\exp(G)$ must divide $k$, we have that $h^k=1$. But then $(hg)^k=g^2$, contradicting the fact that $hg$ and $g^2$ are both contained in $X$. Therefore $C_G(g)=\langle g\rangle$, and has index at most 2 in $G$.

There are two possibilities. Either $G=C_G(g)\cong C_3$, or $G$ is dihedral of order 6. In both cases $g$ has order 3 and $X=G\setminus\{1,g\}$. The fact that $k\equiv 2\mod{3}$ is both necessary and sufficient for $X$ to be an independent set when $G\cong C_3$. This is case (b) of the theorem. However when $G$ is dihedral, if $k$ is odd then each involution of $G$ forms a directed loop in $D_k(G)$, meaning that $D_k(G)$ would have five directed cycles. Therefore $k$ must be even and congruent to 2 modulo 3; that is, $k\equiv 2\mod{6}$. A quick check shows that when $k\equiv 2\pmod{6}$, the set $G\setminus\{1,g\}$ is indeed an independent set. This is case (c) of the theorem.

The remaining case to consider is where every prime factor of $\exp(G)$ divides $k$. We cannot have $\exp(G)$ dividing $k$, for then $s_k(G)=|G|-1$. So there is a prime $p$ and a positive integer $\ell\geq 2$ such that $p^{\ell-1}$ divides $k$, $p^\ell$ does not divide $k$, and $p^\ell$ divides $\exp(G)$. This guarantees that $G$ contains an element $x$ of order $p^\ell$. Let $g=x^k$, an element of order $p$. Then $g\neq 1$ and $D_k(G)$ contains directed edges $(x,g)$ and $(x^{p+1}, g)$. Hence $X=G\setminus\{1,g\}$ is the only potential independent set. If $g^2\not\in\{1,g\}$, then $D_k(G)$ also contains directed edges $(x^2, g^2)$ and $(x^{p+2}, g^2)$. But $\{x^2,x^{p+2},g^2\}\subseteq X$, a contradiction. Therefore $g^2\in\{1,g\}$. Since g is non-trivial, $g^2=1$. So $g$ is an involution, $p=2$, and $k$ is even. We have shown that if $q$ is any odd prime divisor of $\exp(G)$ such that $q^{\ell'-1}$ divides $k$ and $q^{\ell'}$ divides $\exp(G)$ then $q^{\ell'}$ divides $k$. Hence the maximal odd factor of $\exp(G)$ divides $k$.

If $G$ contains an element $y$ of order $2^{\ell+1}$, then $D_k(G)$ has a directed path $y, y^k, 1$. Since $y$ has order $2^{\ell+1}$ and $y^k$ has order 4, neither of these elements are $1$ or $g$, hence both lie in $X$, contradicting the fact that $X$ is an independent set. Therefore $G$ does not contain an element of order $2^{\ell+1}$, meaning that $\exp(G)$ divides $2k$.

Any conjugate $g'$ of $g$ has the same in-degree as $g$, so there is an edge $(x', g')$, where $x'$ is the corresponding conjugate of $x$. If $g'\neq g$, then both $x'$ and $g'$ lie in $X$, a contradiction. Therefore $g$ is central in $G$, and so $\langle g\rangle$ is a normal subgroup. Finally, since $X$ is an independent set, we have $h^k\in\langle g\rangle$ for all $h\in G$. Hence $H:=G/\langle g\rangle$ has exponent dividing $k$. We have shown that in this case $k$ is even and $G$ is a central extension of $H$ (a group of exponent dividing $k$) by $\langle g\rangle\cong C_2$. The fact that $G$ does not have a $k$-power-avoiding set of size $|G|-1$ already implies that $\exp(G)\neq k$. This is case (d) of the theorem.
\end{proof}

\subsection{Groups $G$ with $s_k(G)=|G|-2$ when $k$ is prime}
\label{sec:k=2}

We can use Theorem~\ref{thm:c=2} to derive an even more explicit classification of finite groups $G$ with $s_k(G)=|G|-2$ if we restrict to the case when $k$ is prime. When $k$ is an odd prime, Theorem~\ref{thm:c=2} immediately gives the following.

\begin{corollary}
\label{cor:k_odd_prime}
Let $G$ be a finite group and $p$ be an odd prime. Then $s_p(G)=|G|-2$ if and only if either $G$ is isomorphic to the direct product of $C_2$ with a group of exponent $p$, or $p\equiv 2\mod{3}$ and $G\cong C_3$.
\end{corollary}

When $k=2$, more groups occur. 
Recall that a $p$-group $E$ is \emph{extra-special} if $Z(E)$ has order $p$, and $E/Z(E)$ is non-trivial and elementary abelian. There are exactly two isomorphism classes of extra-special groups of order $p^n$ when $n$ is odd and $n\geq 3$, and no extra-special groups for other values of $n$. (See, for example, Gorenstein~\cite[Section~5.5]{Gorenstein}.)
We have the following result.

\begin{theorem}
\label{thm:k=2_classification}
Let $G$ be a finite group. Then $s_2(G)=|G|-2$ if and only if one of the following holds.
\begin{enumerate}
\item[(a)] $G\cong C_3$;
\item[(b)] $G$ is dihedral of order $6$;
\item[(c)] $G\cong (C_2)^j\times C_4$ for some non-negative integer $j$;
\item[(d)] $G\cong (C_2)^j\times E$ for some non-negative integer $j$ and some extra-special $2$-group $E$;
\item[(e)] $G\cong (C_2)^j\times(C_4*E)$ for some non-negative integer $j$ and some extra-special $2$-group $E$, where $C_4*E$ is the central product of $E$ with $C_4$ (of order $|2|E|$).
\end{enumerate}
\end{theorem}

\begin{proof}
It is not hard to check that all the groups $G$ listed in Theorem~\ref{thm:k=2_classification} do indeed have the property that $s_2(G)=|G|-2$. It suffices to show that there are no more groups with this property.

Let $G$ be a group with $s_2(G)=|G|-2$. By Theorem~\ref{thm:c=2} in the case $k=2$, either $G\cong C_3$, $G$ is dihedral of order $6$, or $G$ satisfies the conditions in case~(d) of Theorem~\ref{thm:c=2}. So it suffices to show that when $G$ satisfies the conditions in case~(d) of Theorem~\ref{thm:c=2} for $k=2$, we have that $G\cong (C_2)^j\times C_4$, $G\cong (C_2)^j\times E$ or $G\cong (C_2)^j\times(C_4*E)$, where $j$ is a non-negative integer and where $E$ is an extra-special $2$-group.

Let $G$ be a finite group satisfying the conditions of case~(d) of Theorem~\ref{thm:c=2} for $k=2$. So $G$ is not of exponent $2$, but $G$ contains a central subgroup $Z=\{1,g\}$ of order $2$ such that $H:=G/Z$ has exponent $2$. Since $H$ has exponent $2$, we see that $H$ is elementary abelian. The exponent of $G$ divides $4$ and is not $2$, so $G$ has exponent $4$. It suffices to show that $G$ is one of the groups described in parts~(c), (d) and~(e) of Theorem~\ref{thm:k=2_classification}.

Since $G/Z$ is elementary abelian, $x^2\in Z$ for all $x\in G$ and so $G^2\subseteq Z$, where $G^2$ is the subgroup generated by all squares in $G$. Since $G$ is not of exponent $2$, the subgroup $G^2$ is non-trivial. Hence $G^2=Z$.   

Suppose that $G$ is abelian. Since $|G^2|=|Z|=2$ and since $G$ has exponent~$4$, we must have $G\cong (C_2)^j\times C_4$ for some non-negative integer $j$. This is case~(c) of the theorem, as required.

Now suppose that $G$ is non-abelian, so $G'$ is nontrivial. Since $G/Z$ is abelian, $G'\subseteq Z$. As $G'$ is non-trivial, we find that $G'=Z$.

We may write $G=(C_2)^j\times M$ for some non-negative integer $j$, where $M$ is a subgroup which does not have $C_2$ as a direct factor. Since $G$ is non-abelian, so is $M$. So $M'$ is a non-trivial subgroup of $G'$, and hence $M'=Z$. In particular, $Z\subseteq M$. The quotient group $M/Z$ is elementary abelian, as it is a subgroup of $G/Z$.

Since $Z$ is central in $G$, we see that $Z\subseteq Z(M)$. If $Z(M)=Z$, then $M$ is extra-special and so $G$ is described by case~(d) of the theorem, as required. So we may assume that $Z$ is a proper subgroup of $Z(M)$.

Suppose $y\in Z(M)\setminus Z$, and let $Y=\langle y\rangle$. Let $E$ be the subgroup of $M$ containing $Z$ such that $E/Z$ is a complement to $YZ/Z$ in $M/Z$. Since $E/Z$ is a complement to $YZ/Z$, we see that $Y\cap E\subseteq Z$. Since $y$ is central, $\langle y\rangle$ is normal in $M$. Moreover, $E$ is normal in $M$ as $E/Z$ is normal in $M/Z$. If $y$ has order $2$, then $Y\cap Z$ is trivial and so $M=Y\times E$. But then we have a contradiction, since $Y\cong C_2$ and $M$ has no direct factors isomorphic to $C_2$. We may deduce that $Y\cong C_4$.

The argument in the previous paragraph shows that all elements in $Z(M)\setminus Z$ have order $4$. In particular, since $Z$ is a proper subgroup of $Z(M)$, we see that $Z(M)\cong C_4$. So $Z(M)=Y$.

Since $YZ/Z$ has exponent $2$, we see that $Z\cap Y$ is non-trivial and so $Z\subseteq Y$. Since $E/Z$ is a complement to $Y/Z$ in $M/Z$, we deduce that $Y\cap E=Z$. So $M=Y*E$, the central product of $Y$ and $E$ at $Z$.

Since $Y$ is central and $M$ is generated by $Y$ and $E$, we see that $Z(E)\subseteq Z(M)$. But $Z(M)= Y$, so $Z(E)\subseteq Y\cap E=Z$. Since $Z(E)$ is non-trivial, $Z(E)=Z$. In particular, $|Z(E)|=2$. But $E/Z$ is elementary abelian, as $E/Z$ is a subgroup of the elementary abelian group $M/Z$. So $E$ is extra-special. This means that $G\cong (C_2)^j\times (C_4*E)$ where $E$ is extra-special, and we are under case~(e) of the theorem as required.
\end{proof}

Writing $D_8$ and $Q_8$ for the dihedral and quaternion groups of order $8$ respectively, the two extra-special groups $E_1$ and $E_2$ of order $2^{2r+1}$ can be written as iterated central products: $E_1\cong D_8*D_8*\cdots *D_8$ and $E_2\cong Q_8*D_8*\cdots *D_8$ respectively, where there are $r$ groups in each of the products. These groups can be distinguished by counting the number of involutions, so there are exactly two isomorphism classes of groups in part~(d) of Theorem~\ref{thm:k=2_classification} once $j$ and $|E|$ are fixed. It might be slightly surprising to realise that there is only one group in part~(e) of the theorem (up to isomorphism) when $j$ and $|E|$ are fixed, because (we claim) $C_4*E_1\cong C_4*E_2$. To establish the claim, we first note that $C_4*E_1\cong (C_4*D_8)*(D_8*\cdots *D_8)$ and $C_4*E_2\cong (C_4*Q_8)*(D_8*\cdots *D_8)$, so it suffices to show that $C_4*D_8\cong C_4*Q_8$. Choose generators $a, b$ for $D_8$ with $a^4 = 1 = b^2$ and $ba = a^{-1}b$, and a generator $x$ for $C_4$ with $x^2 = a^2$, then $D_8*C_4 = \langle a,b,x\rangle = \langle a, bx, x\rangle$ and $a^4 = 1$, $(bx)^2 = a^2$, $(bx)a = a^{-1}(bx)$ so $\langle a, bx\rangle \cong Q_8$. Thus $D_8 * C_4 \cong Q_8 * C_4$ and our claim is established.

\section{More on the directed $k$th-power graph}
\label{sec:structure}

The directed $k$th power graph $D_k(G)$ of a group $G$ has lots of symmetry. In particular, since any automorphism of $G$ commutes with taking $k$th powers, we have the following theorem.

\begin{theorem}
    \label{thm:automorphisms}
    Let $G$ be a group, let $k$ be a positive integer with $k\geq 2$, and let $D_k(G)$ be the directed $k$th power graph of $G$. Then $\Aut(G)$ acts faithfully as a group of (directed graph) automorphisms of $D_k(G)$.
\end{theorem}

So $\Aut(D_k(G))$ (the group of directed graph automorphisms of $D_k(G)$) has a subgroup isomorphic to $\Aut(G)$. In general, there will be other symmetries. For example, taking any integer $m$ coprime to the exponent $e$ of $G$, the map $x\mapsto x^m$ is not in general a group automorphism of $G$, but is a directed graph automorphism of $D_k(G)$. To see this, note that the map $x\mapsto x^m$ is bijective, as its inverse is the map $x\mapsto x^{m'}$ where $mm'\equiv 1 \bmod e$. (A suitable integer $m'$ exists because $(m,e)=1$.) The map $x\mapsto x^m$ is therefore an isomorphism of directed graphs, since $h=g^k$ implies that $h^m=(g^k)^m=(g^m)^k$ for all elements $g,h\in G$.

We will prove two theorems which set out some of the structure of the directed $k$th power graph, allowing us to provide more automorphisms of $D_k(G)$. We begin with some straightforward observations.

\begin{theorem}
    \label{thm:basic_structure}
    Let $G$ be a finite group, and let $k$ be an integer with $k\geq 2$. Let $D_k(G)$ be the directed $k$th power graph of $G$.
\begin{enumerate}
    \item[(a)] Every connected component of $D_k(G)$ contains exactly one (directed) cycle. If the edges of this cycle are removed, the component becomes a root-directed forest whose roots are the vertices of the cycle.
    \item[(b)] A vertex $b\in D_k(G)$ lies in a directed cycle if and only if the order $s$ of $b$ is coprime to $k$.
    \item[(c)] Let $(b_0,b_1,\ldots,b_{r-1})$ be a directed cycle in $D_k(G)$. The elements $b_i$ in the cycle have the same order $s$ and equal centralisers in $G$.
\end{enumerate}
\end{theorem}
\begin{proof}
The first part of the theorem follows since $D_k(G)$ is constructed by iterating a function (namely the function $x\mapsto x^k$) on a finite set. In more detail, we argue as follows.

Let $v$ be a vertex of $D_k(G)$ and let $[v]=v_0,v_1,v_2,\ldots$ be the unique infinite directed walk starting at $v$. (Note that the walk $[v]$ is indeed unique, since $D_k(G)$ has out-degree $1$.) Since the vertex set $G$ of $D_k(G)$ is finite, this path must eventually enter a cycle. Hence every connected component of $D_k(G)$ contains a directed cycle. Suppose that $u$ and $v$ lie in the same connected component of the undirected graph underlying $D_k(G)$. If $u$ and $v$ are joined by a directed edge in  $D_k(G)$, then there is an edge $u\rightarrow v$ or an edge $v\rightarrow u$ in $D_k(G)$ and so (as sets of vertices) we have $[v]\subseteq [u]$ or $[u]\subseteq [v]$ respectively. Thus $[u]$ and $[v]$ eventually enter the same cycle. Applying this argument repeatedly along a path between any two vertices in the connected component shows that the component contains exactly one cycle.

If there is a cycle in the undirected graph underlying $D_k(G)$ then it must correspond to a directed cycle in $D_k(G)$, because $D_k(G)$ has out-degree $1$. So when we remove the edges in the unique directed cycle in a directed component we obtain a cycle-free graph: a directed forest. The vertices of this directed forest have out-degree $1$ (for vertices outside the cycle) or out-degree $0$ (for vertices in the cycle). The result is a root-directed forest, whose roots are the vertices of the original cycle. Part (a) of the theorem follows.

We now turn to second part of the theorem. Suppose $b\in D_k(G)$ is a vertex of order $s$. If $b$ lies in a directed cycle of length $r$, then $b^{k^r}=b$ and so $b^{k^r-1}=1$. Hence $s$ divides $k^r-1$ and we see that $s$ is coprime to $k$. Conversely, suppose that $s$ is coprime to $k$. So $k$ is invertible modulo $s$, of multiplicative order $r$, say. Since $k^r\equiv 1\bmod s$, we find that $v^{k^r}=v$ and so $v$ lies in a cycle, as required. (Indeed, this cycle has length $r$.) So part (b) of the theorem follows.

We now prove the final part of the theorem. We first observe that for an element $x\in G$ and an integer $j$, we have that $C_G(x)\leq C_G(x^j)$, and the order $|x^j|$ of $x^j$ divides the order $|x|$ of $x$.

Let $(b_0,\ldots,b_{r-1})$ be a cycle in $D_k(G)$. Let $i\in\{1,2,\ldots ,r-1\}$. Since $b_i=b_0^{k^i}$, we see that 
$C_G(b_0)\leq C_G(b_i)$, and
$|b_i|$ divides $ |b_0|$. Since $b_0=b_i^{k^{r-i}}$, we see that
$C_G(b_i)\leq C_G(b_0)$ and 
$|b_0|$ divides $ |b_i|$. Hence
$C_G(b_0)= C_G(b_i)$ and 
$|b_i|=|b_0|$, and the theorem follows.
\end{proof}

\begin{definition}
    Let $D_k(G)$ be the directed $k$th power graph of a group $G$, and suppose $b\in D_k(G)$ is a vertex contained in a cycle $C$ in $D_k(G)$. We define $T_b$ to be the set of vertices $v\in D_k(G)$ such that there exists a directed path from $v$ to $b$ that avoids $C\setminus \{b\}$. We give $T_b$ the structure of a root-directed tree with root $b$ by adding a directed edge from $g$ to $g^k$ whenever $g\in T_b\setminus\{b\}$. (In other words, we take the directed graph on $T_b$ induced by $D_k(G)$, removing the loop at $b$ if necessary.)
\end{definition}
\begin{definition}
    The \emph{primary component} of $D_k(G)$ is the set $T_1$ of vertices $v$ such that there is a directed path from $v$ to the identity $1$ of $G$.
\end{definition}

For example, in Figure~\ref{fig:graph_independent_set} the vertex set of the primary component $T_1$ is $\{1,g^5,g^{10},g^{15}\}$, and $T_{g^4}$ has vertex set $\{g,g^2,g^4,g^{11}\}$. The subgraph $T_{g^7}$ is not defined, since $g^7$ is not in a cycle.

We note that a vertex $v$ lies in the primary component $T_1$ of $D_k(G)$ if and only if the order of $v$ divides some power of $k$. If we regard $D_k(G)$ as an undirected graph by ignoring edge orientations, we see that $T_1$ is a connected component of $D_k(G)$, which explains our use of the word `component' here.

We remark that if we remove the edges in a cycle $C=(b_0,b_1,\ldots ,b_{r-1})$ from the connected component containing $C$, we obtain the root-directed forest which is the disjoint union of the trees $T_{b_i}$ for $0\leq i\leq r-1$.

For any subgroup $H$ of $G$, we see that $H\cap T_1$ is a root-directed subtree of $T_1$ with the same root $1$. This follows since $1\in H$ and since the map $x\mapsto x^k$ maps any element $x\in H$ into $H$.

\begin{theorem}
    \label{thm:primary_tree}
    Let $G$ be a group, and let $k$ be an integer with $k\geq 2$. Let $(b_0,b_1,\ldots ,b_{r-1})$ be a directed cycle in $D_k(G)$. Let $b$ be an element of the cycle; without loss of generality we take $b=b_0$. Then $T_1\cap C_G(b)$ and $T_{b}$ are isomorphic as root-directed trees.
\end{theorem}
Before we prove Theorem~\ref{thm:primary_tree}, we note that, by Theorem~\ref{thm:basic_structure}(c), the centralisers $C_G(b_i)$ of the elements in a directed cycle $(b_0,b_1,\ldots,b_{r-1})$ are equal. So Theorem~\ref{thm:primary_tree} implies that all the trees $T_{b_i}$ which are attached to our directed cycle are isomorphic. Moreover, $T_1$ is a `universal object' for the trees attached to cycles in $D_k(G)$. For the example provided in Figure~\ref{fig:graph_independent_set}, we see that the trees attached to each element of the cycle $(g^4,g^8,g^{16},g^{12})$ are all isomorphic to $T_1$. The fact that all of $T_1$ is involved comes from the fact that the group $G$ is abelian in this example.

Theorems~\ref{thm:basic_structure} and~\ref{thm:primary_tree} show that $D_k(G)$ will generally have many automorphisms: there are automorphisms that rotate each cycle, whilst leaving other cycles fixed, and any automorphism of a directed tree $T_b$ can be extended to an automorphism of $D_k(G)$ (acting as the identity outside $T_b$).

\begin{proof}[Proof of Theorem~\ref{thm:primary_tree}]
We regard the subscripts of the elements $b_i$ in our cycle as being taken modulo $r$. So we may write $b_{i+1}=b_i^k$ for all integers $i$. In particular, $b_{-j}$ is the unique element of the cycle such that $b_{-j}^{k^j}=b_0=b$. 

Let $x\in T_1\cap C_G(b)$. We define the \emph{height} $h$ of $x$ to be the smallest positive integer such that $x^{k^h}=1$. So $h$ is the smallest power of $k$ divisible by the order of $x$. In particular, the height of an element is determined by its order, and elements of different heights cannot have the same order. Note that when $x$ has positive height $h$, the element $x^k$ has height $h-1$. So the order of $x^k$ is strictly less than the order of $x$ when $x\not=1$.

We provide a map $\theta:T_1\cap C_G(b)\rightarrow T_b$, which we will prove is an isomorphism, as follows. We define, for every $x\in T_1\cap C_G(b)$ of height $h$,
\[
\theta(x)=xb_{-h}.
\]
We need to show that $\theta$ is well defined, by showing that $\theta(x)\in T_b$.

For any $g\in C_G(b)$ and any non-negative integer $i$, we see that $g\in C_G(b)=C_G(b_i)$, by Theorem~\ref{thm:basic_structure}(c). Hence for all $x\in C_G(b)$ and all non-negative integers $j$, 
\begin{equation}
\label{eqn:homomorphism}
\left(\theta(x^{k^j})\right)^k=\left(x^{k^j}b_{j-h}\right)^k =x^{k^{j+1}}b_{j-h}^k=x^{k^{j+1}}b_{(j+1)-h}=\theta(x^{k^{j+1}}).
\end{equation}
So for all $x\in T_1\cap C_G(b)$, we have a directed walk
\[
\theta(x)\rightarrow \theta(x^k)\rightarrow\cdots \rightarrow \theta(x^{k^h})=x^{k^{h}}b_0=b
\]
from $\theta(x)$ to $b$ in $D_k(G)$. The sequence of elements $x,x^k,x^{k^2},\ldots,x^{k^h}$ have strictly decreasing height and so strictly decreasing order; all orders divide $k^h$. By Theorem~\ref{thm:basic_structure}, the elements $b_{-j}$ have equal order $s$, where $s$ is coprime to~$k$. Since $x^{k^j}$ and $b_{-j}$ have coprime orders and commute, 
\[
|\theta(x^{k^j})|=|x^{k^j}b_{j-h}|=|x^{k^j}||b_{j-h}|=|x^{k^j}|s
\]
and so the orders of the elements in our directed walk have strictly decreasing order. In particular, the elements in our walk are distinct: we have a directed path. Moreover, since $x$ has height $h$, the orders of the elements $x^{k^j}b_{j-h}$ are not coprime to $k$ when $j< h$, so the only member of our cycle involved in our path is $b$. Hence $\theta(x)\in T_b$, and $\theta$ is well-defined.

For $x\in C_G(b)\cap T_1$, we have $\theta(x^k)=\theta(x)^k$, by~\eqref{eqn:homomorphism} in the special case when $j=1$. So $\theta$ is a directed graph homomorphism. 

Suppose that $x,x'\in T_1\cap C_G(b)$ are such that $\theta(x)=\theta(x')$. We have $|\theta(x)|=|x|s$ and $|\theta(x')|=|x'|s$, so $|x|=|x'|$. In particular, $x$ and $x'$ have the same height $h$ and hence $xb_{-h}=\theta(x)=\theta(x')=x'b_{-h}$. So $x=x'$. We have shown that $\theta$ is injective.

It remains to show that $\theta$ is surjective. Let $y\in T_b$. Then there is a directed path
\[
y=y_0\rightarrow y_1\rightarrow\cdots \rightarrow y_h=b
\]
in $D_k(G)$, such that the elements $y_i$ for $i<h$ do not lie in our cycle, so have order greater than $s$. Define $x=yb_{-h}^{-1}$. Now $b_{-h}$ is a power of $y$, since $b_{-h}$ is a power of $b$ and since $b=y_h=y^{k^h}$. Since $x=yb_{-h}^{-1}$, we deduce that $x$ is a power of $y$ and, in particular, $x\in C_G(b)$. Because all powers of $y$ commute,
\[
y_i=y^{k^i}=(xb_{-h})^{k^i}=x^{k^i}b_{-h}^{k^i}=x^{k^i}b_{i-h}.
\]
Now $y_h=b$, so the formula above in the case $i=h$ implies that $b=y_h=x^{k^h}b_{0}=x^{k^h}b$, and therefore $x^{k^h}=1$. Hence $x\in T_1$ and has height $h'$ for some $h'\leq h$. But
$|y_i|=|x^{k^i}||b_{i-h}|=|x^{k^i}|s$, the first equality following by virtue of the fact that $x^{k^i}$ and $b_{i-h}$ are commuting elements of coprime order. Hence $y_{h'}$ has order $s$, and therefore lies on a cycle. Since the elements $y_i$ for $i<h$ do not lie on a cycle, we see that $h'\geq h$. Hence $h'=h$ and $x$ has height $h$. But then $\theta(x)=xb_{-h}=y$. So $\theta$ is surjective and the theorem is proved.
\end{proof}

\end{document}